\documentclass[11pt]{amsart}

\usepackage[T1]{fontenc}
\usepackage[latin1]{inputenc}
\usepackage[english]{babel}
\usepackage{amsmath,amssymb,amsthm} 
\usepackage{enumitem}
\usepackage{hyperref}

\newtheorem{theorem}{Theorem}[section]
\newtheorem{lemma}[theorem]{Lemma}
\newtheorem{proposition}[theorem]{Proposition}
\newtheorem{corollary}[theorem]{Corollary}

\theoremstyle{definition}

\newtheorem{remark}[theorem]{Remark}
\newtheorem*{remark*}{Remark}

\newcommand{\SA}{\textup{SA}}
\newcommand{\cb}{\textup{cb}}
\newcommand{\Han}{\operatorname{Han}}
\newcommand{\supp}{\operatorname{supp}}
\newcommand{\dist}{\operatorname{dist}}

\renewcommand{\MR}[1]{}

\title[von Neumann's inequality on the polydisc]{On von Neumann's inequality on the polydisc}
\author[M. Hartz]{Michael Hartz}
\address{Fachrichtung Mathematik, Universit\"at des Saarlandes, 66123 Saarbr\"ucken, Germany}
\email{hartz@math.uni-sb.de}
\thanks{The author was partially supported by the
Emmy Noether Program of the German Research Foundation (DFG Grant 466012782)}


\begin{document}

\begin{abstract}
  Given a $d$-tuple $T$ of commuting contractions on Hilbert space and a polynomial $p$ in $d$-variables,
  we seek upper bounds for the norm of the operator $p(T)$.
  Results of von Neumann and And\^o show that if $d=1$ or $d=2$, the upper bound $\|p(T)\| \le \|p\|_\infty$,
  holds, where the supremum norm is taken over the polydisc $\mathbb{D}^d$.
  We show that for $d=3$, there exists a universal constant $C$ such that $\|p(T)\| \le C \|p\|_\infty$
  for every homogeneous polynomial $p$. We also show that for general $d$ and arbitrary polynomials, the norm $\|p(T)\|$
  is dominated by a certain Besov-type norm of $p$.
\end{abstract}

\maketitle

\section{Introduction}

A famous inequality of von Neumann \cite{Neumann51} shows that if
$T$ is a contraction on a Hilbert space $\mathcal{H}$,
then
\begin{equation*}
  \|p(T)\| \le \sup_{z \in \mathbb{D}} |p(z)|
\end{equation*}
for every polynomial $p \in \mathbb{C}[z]$. This inequality is the basis 
of an important connection between operator theory and complex analysis;
see for instance \cite{AMY20,SFB+10}.
And\^o \cite{Ando63} extended von Neumann's inequality to two variables. His inequality shows
that if $T = (T_1,T_2)$ is a pair of commuting contractions on Hilbert space, then
\begin{equation*}
  \|p(T)\| \le \sup_{z \in \mathbb{D}^2} |p(z)|
\end{equation*}
for all polynomials $p \in \mathbb{C}[z_1,z_2]$. However, the corresponding inequality for three or more commuting contractions is false, as examples of Kaijser--Varopoulos \cite{Varopoulos74} and Crabb--Davie \cite{CD75} show.
More background information can be found in the books \cite{Paulsen02,Pisier01} and in \cite{BB13}.

Even though the counterexamples to von Neumann's inequality in three variables were discovered in the seventies,
many questions surrounding this phenomenon were only answered recently or remain open.
For instance, the smallest dimension of a Hilbert space on which there exist counterexamples in three variables
was only determined a few years ago: there exist counterexamples in dimension four due to Holbrook \cite{Holbrook01},
whereas Knese showed that the inequality holds in dimension three or less \cite{Knese16}; this relies on a result in complex geometry due to Kosi\'nski \cite{Kosinski15}.
See also \cite{CD13} for a related counterexample in dimension three.

Remarkably, it is still not known if von Neumann's inequality for three commuting contractions holds up to a constant; see for instance \cite[Chapter 1]{Pisier01} and \cite[Chapter 5]{Paulsen02} for a detailed discussion of this problem.
Part of the difficulty comes from the lack of a convenient model for tuples of commuting contractions,
unlike in the setting of operator tuples associated with the Euclidean ball; see \cite{Hartz17}.
See also \cite{MT10} for a non-commutative approach.

To study whether von Neumann's inequality holds up to a constant, one defines $C(d,n) \in [1,\infty)$ to be the smallest constant such that
\begin{equation*}
  \|p(T)\| \le C(d,n) \sup_{z \in \mathbb{D}^d} |p(z)|
\end{equation*}
holds for every homogeneous polynomial $p \in \mathbb{C}[z_1,\ldots,z_d]$ of degree $n$ and every $d$-tuple $T$ of commuting contractions on Hilbert space.
By von Neumann's and And\^o's inequalities, $C(1,n) = C(2,n) = 1$ for all $n \in \mathbb{N}$.
Dixon \cite{Dixon76} showed that for $n \ge 2$,
\begin{equation}
  \label{eqn:Dixon_upper_bound}
  C(d,n) \le G_{\mathbb{C}} (3 d)^{(n-2)/2} (2 e)^n,
\end{equation}
where $G_{\mathbb{C}}$ is the complex Grothendieck constant, which satisfies $G_{\mathbb{C}} < \frac{3}{2}$;
see \cite{Haagerup87}.
He also proved that for fixed $n$ as $d \to \infty$, this estimate is not too far from optimal.
Explicitly, up to a constant depending on $n$, he established a lower bound for $C(d,n)$ of the form
$d^{\frac{1}{2} \lfloor (n-1)/2 \rfloor}$.
See also \cite{GR18} for some recent work on determining the value of $\lim_{d \to \infty} C(d,2)$.
However, as Dixon already remarked, the asymptotic behavior of $C(d,n)$ as $d \to \infty$ does not directly bear on the question
of whether von Neumann's inequality for $d$ commuting contractions holds up to a constant. Indeed, for this question,
the behavior of $C(d,n)$ for fixed $d$ as $n \to \infty$ is relevant.

For fixed $d$, Dixon's upper bound \eqref{eqn:Dixon_upper_bound} is exponential in the degree $n$. However,
it is easy to obtain upper bounds for $C(d,n)$ that are polynomial in $n$. For instance,
if $p(z) = \sum_{\alpha} \widehat{p}(\alpha) z^\alpha$ is homogeneous of degree $n$,
then for any $d$-tuple of commuting contractions, the Cauchy--Schwarz inequality shows that
\begin{align*}
  \|p(T)\| \le \sum_{\alpha} |\widehat{p}(\alpha)| &\le \binom{d+n-1}{d-1}^{1/2} \Big( \sum_{\alpha} |\widehat{p}(\alpha)|^2 \Big)^{1/2} \\
                                                   &\le \binom{d+n-1}{d-1}^{1/2} \sup_{z \in \mathbb{D}^d} |p(z)|.
\end{align*}
Here, the binomial coefficient is the dimension of the space of homogeneous polynomials in $d$ variables of degree $n$.
This gives the upper bound $C(d,n) \lesssim_d (n+1)^{(d-1)/2}$. Here and in the sequel, we write $f \lesssim_d g$ to mean
that there exists a constant $C$ depending only on $d$ such that $f \le C g$.

Our first main result gives an upper bound that is polylogarithmic in the degree $n$
and in particular yields that $C(3,n)$ is uniformly bounded in $n$.

\begin{theorem}
  \label{thm:main1}
  Let $d \ge 3$. Then for all $n \ge 1$,
  \begin{equation*}
    C(d,n) \lesssim_d (\log(n+1))^{d-3}.
  \end{equation*}
  In particular, $\sup_{n} C(3,n) < \infty$.
\end{theorem}
This result will be proved in Corollary \ref{cor:C(d,n)_upper_bound}.
The proof yields the crude explicit upper bound $C(3,n) \le 121$, see Remark \ref{rem:explicit}, but no attempt was made to optimize the numerical bound.

It is important to keep in mind that $\sup_n C(3,n) < \infty$ does not imply that von Neumann's inequality
for three commuting contractions
holds up to a constant, as $C(3,n)$ is only defined using homogeneous polynomials.
Nonetheless, the known counterexamples to von Neumann's inequality in three variables
all use homogeneous polynomials, and show in particular that $C(3,2) > 1$. (The best known
lower abound appears to be $C(3,2) \ge \frac{1}{3} \sqrt{\frac{35+13\sqrt{13}}{6}} \approx 1.23$; see \cite[Proposition 6.1]{GKW13}.)

In \cite[Chapter 4]{Pisier01}, Pisier gives an exposition of work of Daher, which studies von Neumann's inequality
for tuples of commuting $N \times N$ matrices; see \cite[Corollary 4.21]{Pisier01}.
In particular, modifying arguments of Bourgain \cite{Bourgain86}, this work
shows that von Neumann's inequality holds up to a factor of $(\log(N+1))^d$ in this context.
Such modifications of Bourgain's arguments can also be used to establish the upper bound $C(d,n) \lesssim_d (\log(n+1))^{d}$.

Given a not necessarily homogeneous polynomial $p \in \mathbb{C}[z_1,\ldots,z_d]$, the Schur--Agler norm of $p$ is defined to be
\begin{equation*}
  \|p\|_{\SA} = \sup \{ \|p(T)\|\},
\end{equation*}
where the supremum is taken over all $d$-tuples $T$ of commuting contractions on Hilbert space.
It is natural to seek function theoretic upper bounds for the Schur--Agler norm.
To this end, recall that if $f: \mathbb{D}^d \to \mathbb{C}$ is holomorphic,
the radial derivative of $f$ is $(R f)(z) = \sum_{j=1}^d z_j \frac{\partial f}{\partial z_j}(z)$.
Let also write $f_r(z) = f(r z)$ and use $\|\cdot\|_\infty$ to denote the supremum norm on $\mathbb{D}^d$.
We now have the following Besov-type upper bound for the Schur--Agler norm.

\begin{theorem}
  \label{thm:main2}
  Let $d \ge 3$. Then for all $p \in \mathbb{C}[z_1,\ldots,z_d]$,
  \begin{equation*}
    \|p\|_{\SA} \lesssim_d
    |p(0)| + \int_{0}^1 \|(R p)_r\|_\infty \Big( \log\Big(\frac{1}{1-r} \Big) \Big)^{d - 3} \, dr.
  \end{equation*}
\end{theorem}
This result will be proved in Corollary \ref{cor:Besov_upper_bound}.

The idea to use Besov norms in the context of functional calculi already appeared in the seminal work of Peller \cite{Peller82} on polynomially bounded operators. Indeed, this article takes inspiration from Peller's work and subsequent works, for instance of Vitse \cite{Vitse05} and Schwenninger \cite{Schwenninger16}.

Our methods also yield some results about the Schur--Agler norm with constant $1$.
It was shown by Knese \cite{Knese11a} that for certain rational inner functions $f: \mathbb{D}^3 \to \mathbb{C}$,
one can find a monomial $q$ such that $\|q f\|_{\SA} = \|f\|_\infty$.
(See the disussion preceding Corollary \ref{cor:Besov_upper_bound} for the definition of Schur--Agler
norm of a holomorphic function.) We will show an asymptotic version of Knese's theorem for polynomials.
Grinshpan, Kaliuzhnyi-Verbovetskyi and Woerdeman proved that there exist polynomials $p \in \mathbb{C}[z_1,z_2,z_3]$
such that $\|z_3 p\|_{\SA} < \|p\|_{\SA}$; see \cite[Theorem 2.3]{GKW14}.
They also asked if for every polynomial $p$ satisfying $\|p\|_\infty < \|p\|_{\SA}$,
there exists a monomial $q$ such that $\|q p\|_{\SA} < \|p\|_{\SA}$;
see \cite[Problem 2.4]{GKW14}.
The following result answers this question in the affirmative;
it also gives another proof of the existence of a polynomial $p \in \mathbb{C}[z_1,z_2,z_3]$ with $\|z_3 p\|_{\SA} < \|p\|_{\SA}$.

\begin{theorem}
  \label{thm:main3}
  Let $d \ge 3$ and let $p \in \mathbb{C}[z_1,\ldots,z_d]$. Then
  \begin{equation*}
    \lim_{m \to \infty} \|(z_3 \cdot \ldots \cdot z_d)^m p\|_{\SA} = \|p\|_\infty.
  \end{equation*}
\end{theorem}

This result will be proved in Theorem \ref{thm:SA_sup_asympt}.

On our way to establishing Theorems \ref{thm:main1} and \ref{thm:main2}, we first study a version of the one-variable
von Neumann inequality, but for polynomials with operator coefficients satisfying a commutativity hypothesis; see Section \ref{sec:op_coeff} for precise details. The idea to use one-variable polynomials with operator coefficients to study the scalar von Neumann inequality in several
variables already appears in work of Daher; see \cite[Chapter 4]{Pisier01}.
In Theorem \ref{thm:OVNI_sharp}, which may be of independent interest, we establish fairly precise upper and lower bounds for von Neumann's inequality with operator coefficients.
However, we will see in Propositions \ref{prop:Hankel_constant_1} and \ref{prop:Hankel_general_constant}
that the operators yielding logarithmic lower bounds in von Neumann's inequality with operator coefficients
satisfy von Neumann's inequality for $d$-tuples up to a constant.

The remainder of this article is organized as follows. Section \ref{sec:op_coeff} contains the material on von Neumann's inequality with operator coefficients. Section \ref{sec:polydisc} deals with von Neumann's inequality on the polydisc. Moreover, in Appendix \ref{sec:Besov}, we collect a few basic facts about analytic Besov spaces.

\subsection*{Acknowledgements} The author is grateful to Bernhard Hack for pointing out the work of Markus Haase \cite{Haase11}. The author would also like to thank the anonymous referee for a close reading,
for numerous valuable comments and in particular for pointing out reference \cite{GKW14},
which led to the formulation of Theorem \ref{thm:main3} for general $d \ge 3$, instead of only $d=3$.

\section{Polynomials with operator coefficients}
\label{sec:op_coeff}

In this section, we consider operator-valued polynomials of the form 
\begin{equation*}
  p(z) = \sum_{k=m}^n A_k z^k,
\end{equation*}
where each $A_k \in B(\mathcal{H})$.
We call such polynomials $(m,n)$-band-limited.
Given another operator $T \in B(\mathcal{H})$, we ``evaluate'' the polynomial
at $T$ as follows:
\begin{equation*}
  p(T) = \sum_{k=m}^n A_k T^k.
\end{equation*}
Notice that the product is not the tensor product, but composition of operators.
(Using the tensor product, one obtains von Neumann's inequality with constant $1$ by Sz.-Nagy's dilation theorem.)

We are mostly concerned with the case when each $A_k$ commutes with $T$.
Equivalently,
$p(z)$ commutes with $A_k$ for all $z \in \mathbb{D}$.
We seek bounds on $\|p(T)\|$ in terms of $\sup_{z \in \mathbb{D}} \|p(z)\|$.
In the case $m=0$, this problem was already studied by Daher and Pisier; see \cite[Chapter 4]{Pisier01}.
To get a feeling for the problem, we first consider the easier case without
any assumption on commutation.

Throughout, we will use $\mathbb{T}$ to denote the unit circle, and we write 
\begin{equation*}
  \int_{\mathbb{T}} f(z) \frac{|dz|}{2 \pi} = \int_{0}^{2 \pi} f(e^{it}) \frac{dt}{2 \pi}.
\end{equation*}

\begin{proposition}
  Let $p(z) = \sum_{k=m}^n A_k z^k$ be a polynomial with operator coefficients and let $\|T\| \le 1$. Then
  \begin{equation*}
    \Big\| \sum_{k=m}^n A_k T^k \Big\| \le \sqrt{n-m+1} \sup_{z \in \mathbb{D}} \| p(z)\|.
  \end{equation*}
  Moreover, the factor $\sqrt{n-m+1}$ is best possible in the sense that for all $1 \le m \le n$,
  there exists a non-zero choice of $A_m,\ldots,A_n$ and $T$ such that equality holds.
\end{proposition}

\begin{proof}
  We may normalize so that $\sup_{z \in \mathbb{D}} \|p(z)\| = 1$.
  Then
  \begin{equation*}
    I \ge \int_{\mathbb{T}} p(z) p(z)^* \frac{|dz|}{2 \pi} = \sum_{k=m}^n A_k A_k^*.
  \end{equation*}
  Hence the row $
  \begin{bmatrix}
    A_m & \cdots & A_n
  \end{bmatrix}$ is a contraction. Then
  \begin{equation*}
    \sum_{k=m}^n A_k T^k =
    \begin{bmatrix}
      A_m & \cdots & A_n
    \end{bmatrix}
    \begin{bmatrix}
      T^m \\ T^{m+1} \\ \vdots \\ T^n
    \end{bmatrix}.
  \end{equation*}
  The column has norm at most $\sqrt{n-m+1}$, from which the upper bound follows.

  To see that equality may hold, let $A_k = E_{1,k+1}$ (matrix units on $\ell^2$) and let $T$ be the unilateral shift. Then
  \begin{equation*}
    \sum_{k=m}^n A_k T^k e_1 = \sum_{k=m}^n E_{1, k+1} e_{k+1} = (n-m+1) e_1,
  \end{equation*}
  hence $\| \sum_{k=m}^n A_k T^k \| \ge n-m+1$,
  but $p(z) p(z)^* = \sum_{k=m}^n E_{1,1} |z|^{2 k}$, so that $\|p\|_\infty = \sqrt{n-m+1}$.
\end{proof}

\begin{remark}
  The proof of the upper bound in fact applies to power bounded operators.
\end{remark}

Recall that $T$ is said to doubly commute with $A$ if $T$ commutes with $A$ and $A^*$.
In the doubly commuting case, one obtains the inequality with constant $1$.
It appears that this was first shown by Arveson and Parrott (unpublished) and Mlak \cite{Mlak71}.

\begin{proposition}
  Let $p(z) = \sum_{k=0}^n A_k z^k$ be a polynomial with operator coefficients and let $T \in B(\mathcal{H})$ with
  $\|T\| \le 1$. If $T$ doubly commutes with each $A_k$, then
  \begin{equation*}
    \Big\| \sum_{k=0}^n A_k T^k \Big\| \le \sup_{z \in \mathbb{D}} \|p(z)\|.
  \end{equation*}
\end{proposition}

\begin{proof}
  We use a small modification of a proof of von Neumann's inequality due to Heinz;
  see \cite{Heinz52} and also \cite[Exercise 2.15]{Paulsen02}.
  We may assume that $\|T\| < 1$. For $z \in \overline{\mathbb{D}}$, consider
  the Poisson-type kernel
  \begin{equation*}
    P(z,T) = (1 - z T^*)^{-1} + (1 - \overline{z} T)^{-1} - I.
  \end{equation*}
  A simple computation shows that $P(z,T) \ge 0$ for all $z \in \overline{\mathbb{D}}$
  and that
  \begin{equation*}
    \sum_{k=0}^n A_k T^k =  \int_{\mathbb{T}} p(z) P(z,T) \frac{|dz|}{2 \pi} =
    \int_{\mathbb{T}} P(z,T)^{1/2} p(z) P(z,T)^{1/2} \frac{|dz|}{2 \pi},
  \end{equation*}
  where the second equality follows from the doubly commuting assumption.
  Now, the map
  \begin{equation*}
    \Phi: C(\mathbb{T},B(\mathcal{H})) \to B(\mathcal{H}), \quad
    f \mapsto \int_{\mathbb{T}} P(z,T)^{1/2} f(z) P(z,T)^{1/2} \frac{|dz|}{2 \pi}
  \end{equation*}
  is unital and completely positive, hence completely contractive; see for instance \cite[Proposition 3.2]{Paulsen02}. In particular, $\|p(T)\| = \|\Phi(p)\| \le \|p\|_\infty$.
\end{proof}

To deal with the singly commuting case, we require the following routine application of the Cauchy--Schwarz inequality.

\begin{lemma}
  \label{lem:integral_cb}
  Let $X$ be a compact Hausdorff space, let $\mu$ be a Borel probability measure on $X$
 and let $K,L: X \to B(\mathcal{H})$ and let $f: X \to {B}(\mathcal{H})$ be norm continuous. Then
  \begin{align*}
    &\left\| \int_X K(x) f(x) L(x) \, d \mu(x) \right\| \\
    \le
    &\left\| \int_X K(x) K(x)^* \, d \mu(x) \right\|^{1/2}
    \left\| \int_X L(x)^* L(x) \, d \mu(x) \right\|^{1/2}
    \, \sup_{x \in X} \|f(x)\|.
  \end{align*}
\end{lemma}

\begin{proof}
  Let $\xi,\eta \in \mathcal{H}$ be unit vectors and assume without loss of generality that $\sup_{x \in X} \|f(x)\| = 1$.
  Applying the Cauchy--Schwarz inequality to the positive semi-definite sesquilinear form
  on $C(X,\mathcal{H})$ defined by
  \begin{equation*}
    (g,h) \mapsto \int_{X} \langle g(x), h(x) \rangle \, d \mu(x),
  \end{equation*}
  we find that
  \begin{align*}
    &\left| \int_{X} \langle K(x) f(x) L(x) \xi, \eta \rangle d \mu(x) \right|^2
    = \left| \int_X \langle f(x) L(x) \xi, K(x)^* \eta \rangle d \mu(x) \right|^2 \\
    \le &\int_{X} \langle f(x)^* f(x) L(x) \xi, L(x) \xi \rangle \, d \mu(x)
    \int_{X} \langle K(x)^* \eta, K(x)^* \eta \rangle \, d \mu(x).
  \end{align*}
  Since $\sup_{x \in X} \|f(x)\| = 1$, we have $f(x)^* f(x) \le I$ for all $x \in X$, so the first factor can be estimated by
  \begin{align*}
    \int_{X} \langle f(x)^* f(x) L(x) \xi, L(x) \xi \rangle \, d \mu(x) &\le
    \int_X \langle L(x) \xi, L(x) \xi \rangle d \mu(x) \\ &\le
    \left\| \int_X L(x)^* L(x) d \mu(x) \right\|.
  \end{align*}
  A similar estimate holds for the second factor, thus
  \begin{align*}
    &\left| \int_{X} \langle K(x) f(x) L(x) \xi, \eta \rangle d \mu(x) \right|^2 \\
    \le
    &\left\| \int_X L(x)^* L(x) d \mu(x) \right\|
    \left\| \int_X K(x) K(x)^* d \mu(x) \right\|.
  \end{align*}
  The assertion now follows by taking the supremum over all unit vectors $\xi, \eta \in \mathcal{H}$.
\end{proof}

We now turn to the simply commuting case.
Given natural numbers $0 \le m \le n$, let $K(m,n)$ be the smallest constant such that
\begin{equation*}
  \Big\| \sum_{k=m}^n A_k T^k \Big\| \le K(m,n) \sup_{z \in \mathbb{D}} \Big\| \sum_{k=m}^n A_k z^k \Big\|
\end{equation*}
holds for all operators $A_m,\ldots,A_n \in B(\mathcal{H})$ and all $T \in B(\mathcal{H})$ satisfying
$\|T\| \le 1$ and $T A_k = A_k T$ for all $k=m,m+1,\ldots,n$.
The arguments of Daher and Pisier show that $K(0,n) \simeq \log(n+1) + 1$; see \cite[Chapter 4]{Pisier01}.
Here, we write $f \simeq g$ to mean $f \lesssim g$ and $g \lesssim f$.
We will obtain fairly sharp estimates for $K(m,n)$ in general.

Let $H^2$ denote the classical Hardy space on the disc and let $\Han \subset H^2$ be the space of symbols of bounded Hankel operators $H^2 \to \overline{H^2}$. Thus, if $b \in \Han$,
then we obtain a bounded Hankel operator $H_b: H^2 \to \overline{H^2}$ satisfying
\begin{equation*}
  \langle H_b f \overline{g} \rangle_{\overline{H^2}} = \langle fg ,b \rangle_{H^2}
\end{equation*}
for all polynomials $f,g$. We equip $\Han$ with the norm $\|b\|_{\Han} = \|H_b\|$.
Nehari's theorem (\cite{Nehari57}, see also \cite[Theorem 1.1]{Peller03}) shows that $\Han \cong L^\infty / \overline{H^\infty_0}$ is the dual space of $H^1$ with respect to the Cauchy pairing.
It is known that $\Han$ can be identified
with $\textup{BMOA}$, but we do not require this. More background on Hankel operators can be found in \cite{Peller03}.

Given a function $h \in H^1$, we write $h(z) = \sum_{k=0}^\infty \widehat{h}(k) z^k$ for the Taylor series of $h$.

\begin{proposition}
  \label{prop:band_limited_abstract}
  We have
  \begin{align*}
    K(m,n) &= \inf \{ \|h\|_{H^1}: \widehat{h}(k) = 1 \text{ for } m \le k \le n \} \\
           &= \sup \{ |q(1)| : \|q\|_{\Han} \le 1 \text{ and } \supp \widehat{q} \subset [m,n] \}.
  \end{align*}
\end{proposition}

\begin{proof}
  The second equality follows from duality.
  Indeed, let 
  \begin{equation*}
    M = \{h \in H^1: \widehat{h}(k) = 0 \text{ for } m \le k \le n \}.
  \end{equation*}
  Then the annihilator of $M$ in $\Han$ is $\{q \in \Han: \supp \widehat{q} \subset [m,n]\}$.
  So if $f \in H^1$ is any function with $\widehat{f}(k) = 1$ for $m \le k \le n$, then
  by the Hahn--Banach theorem,
  \begin{align*}
    &\inf\{\|h\|_{H^1} : \widehat{h}(k) = 1 \text{ for } m \le k \le n\}
     = \dist(f,M) \\
    = &\sup \{ |\langle f,q \rangle|: \|q\|_{\Han} \le 1
    \text{ and } \supp \widehat{q} \subset [m,n] \} \\
      = & \sup \{ |q(1)|: \|q\|_{\Han} \le 1 \text{ and } \supp \widehat{q} \subset [m,n] \}.
  \end{align*}

  Next, we prove that $K(m,n)$ is bounded above by the infimum.
  To this end, we use a factorization argument, which already appears in \cite[III.F.18]{Wojtaszczyk91} and \cite[Proposition 4.16]{Pisier01};
  see also \cite{Haase11} for extensions to semigroups on Banach spaces.
  Let $p(z) = \sum_{k=m}^n A_k z^k$ be an operator-valued polynomial
  with $\sup_{z \in \mathbb{D}} \|p(z)\| \le 1$
  and let $T \in B(\mathcal{H})$ be a contraction that commutes with all $A_k$.
  Let $h \in H^1$ satisfy $\widehat{h}(k) = 1$ for $m \le k \le n$.
  We have to show that
  \begin{equation}
    \label{eqn:to_show}
    \|p(T)\| \le \|h\|_{H^1}.
  \end{equation}

  By replacing $T$ with $r T$ for $r < 1$, we may assume that $\sigma(T) \subset \mathbb{D}$.
  There exist $f,g \in H^2$ so that $h=fg$ and $\|f\|_{H^2} \|g\|_{H^2} = \|h\|_{H^1}$.
  Thus, by the commutation hypothesis,
  \begin{equation*}
    p(T) = \int_{\mathbb{T}} p(z) h(\overline{z} T) \, \frac{|dz|}{2 \pi}
    = \int_{\mathbb{T}} f(\overline{z} T) p(z) g(\overline{z} T) \, \frac{|dz|}{2 \pi}.
  \end{equation*}
  By Lemma \ref{lem:integral_cb}, it follows that
  \begin{equation*}
    \|p(T)\| \le \Big\| \int_{\mathbb{T}} f(\overline{z} T) f(\overline{z} T)^* \, \frac{|dz|}{2 \pi} \Big\|^{1/2}
    \Big\| \int_{\mathbb{T}} g(\overline{z} T)^* g(\overline{z} T) \, \frac{|dz|}{2 \pi} \Big\|^{1/2}.
  \end{equation*}
  Using orthogonality, we find that
  \begin{equation*}
    \int_{\mathbb{T}} f(\overline{z} T) f(\overline{z} T)^* \, \frac{|dz|}{2 \pi}
    = \sum_{k=0}^\infty |\widehat{f}(k)|^2 T^k (T^*)^k,
  \end{equation*}
  hence
  \begin{equation*}
    \Big\| \int_{\mathbb{T}} f(\overline{z} T) f(\overline{z} T)^* \, \frac{|dz|}{2 \pi} \Big\|^{1/2}
    \le \|f\|_{H^2}.
  \end{equation*}
  Similarly,
  \begin{equation*}
    \Big\| \int_{\mathbb{T}} g(\overline{z} T)^* g(\overline{z} T) \, \frac{|dz|}{2 \pi} \Big\|^{1/2} \le \|g\|_{H^2}.
  \end{equation*}
  Since $\|f\|_{H^2} \|g\|_{H^2} = \|h\|_{H^1}$, the upper bound \eqref{eqn:to_show} follows.
  
  Finally, we show that $K(m,n)$ is bounded below by the supremum.
  To this end, we use Foguel--Hankel operators; see \cite[Chapter 10]{Paulsen02} for background.
  For $q \in \Han$, let $H_q: H^2 \to \overline{H^2}$ be the Hankel
  operator satisfying $\langle H_q f,\overline{g} \rangle = \langle f g, q \rangle$
  for polynomials $f,g$.
  Let $\zeta$ denote the independent variable on $\mathbb{D}$, and let $M_\zeta: H^2 \to H^2$
  be the shift. Then 
  \begin{equation}
    \label{eqn:Hankel}
    H_q M_\zeta = M_{\overline{\zeta}}^* H_q.
  \end{equation}
  Let
  \begin{equation*}
    T =
    \begin{bmatrix}
      M_\zeta & 0 \\ 0 & M_{\overline{\zeta}}^*
    \end{bmatrix} \in B(H^2 \oplus \overline{H^2}).
  \end{equation*}
  If $q \in \Han$ with $\supp \widehat{q} \subset [m,n]$, define
  \begin{equation*}
    A_k =
    \begin{bmatrix}
      0 & 0 \\
      H_{\widehat{q}(k) \zeta^k} & 0
    \end{bmatrix}.
  \end{equation*}
  Then \eqref{eqn:Hankel} shows that $T$ commutes with each $A_k$.
  Let $p(z) = \sum_{k=m}^n A_k z^k$. Then
  \begin{equation*}
    \|p(z)\| = \Big\|\sum_{k=m}^n \widehat{q}(k) (\zeta \overline{z})^k \Big\|_{\Han}
    \le \|q\|_{\Han}
  \end{equation*}
  for all $z \in \mathbb{T}$ by rotation invariance of $\Han$, and hence for all $z \in \overline{\mathbb{D}}$
  by the maximum principle.
  So if $\|q\|_{\Han} \le 1$, then
  \begin{equation*}
    K(m,n) \ge \Big\| \sum_{k=m}^n A_k T^k \Big\|
    = \Big\|\sum_{k=m}^n H_{\widehat{q}(k) \zeta^k} M_{\zeta^k} \Big\| = \|H_{q(1)}\|
    = |q(1)|.
  \end{equation*}
  This proves the lower bound for $K(m,n)$.
\end{proof}

\begin{remark}
  \label{rem:power_bounded_and_cb}
  (a)
  The proof of the upper bound for $K(m,n)$ in fact works more generally for power bounded operators $T$.
  Thus, using the result of Proposition \ref{prop:band_limited_abstract}, we find that
  if $T$ is power bounded and $A_k$ are operators commuting with $T$, then
  \begin{equation*}
    \Big\| \sum_{k=m}^n A_k T^k \Big\| \le \sup_{n} \|T^n\|^2 K(m,n) \sup_{z \in \mathbb{D}} \Big\| \sum_{k=m}^n A_k z^k \Big\|.
  \end{equation*}
  So in contrast to the scalar von Neumann inequality, contractions do not yield a qualitatively better
  estimate than power bounded operators.

  (b) In the context of the classical inequalities of von Neumann and And\^o,
  it is natural to consider matrices of polynomials because of the connections to dilation theory,
  see for instance \cite[Chapter 7]{Paulsen02} and \cite[Chapter 4]{Pisier01}.
  In the present setting, one can similarly consider $r \times r$ matrices $[p_{ij}]$,
  where each entry $p_{ij}$ is an operator-valued polynomial of the form
  \begin{equation*}
    p_{i j}(z) = \sum_{k=m}^n A_k^{(ij)} z^k,
  \end{equation*}
  and each $A_k^{(ij)} \in B(\mathcal{H})$.
  Such a matrix can be evaluated entry-wise at an operator $T$ that commutes with all coefficients $A_{k}^{(ij)}$.
  Considering such matrices of arbitrary size $r$, one defines a completely bounded version
  $K_{\cb}(m,n)$ of the constant $K(m,n)$.
  However, this setting is actually not more general, and we have $K_{\cb}(m,n) = K(m,n)$.
  Indeed, given an $r \times r$ matrix $[p_{ij}]$ as above and a contraction $T \in B(\mathcal{H})$ commuting with all coefficients
  $A_k^{(ij)}$, let $E_{ij} \in M_r(\mathbb{C})$ be the usual matrix units and define
  \begin{equation*}
    q(z) = \sum_{i,j=1}^r \sum_{k=m}^n (A^{(ij)}_k \otimes E_{ij}) z^k.
  \end{equation*}
  Then $q$ is a polynomial with coefficients in $B(\mathcal{H}) \otimes M_r(\mathbb{C})$,
  which we identify with $B(\mathcal{H}^r)$,
  and $\sup_{z \in \mathbb{D}} \|q(z)\| = \sup_{z \in \mathbb{D}} \|[p_{ij}(z)]\|$.
  The contraction $T \otimes I$ commutes with the coefficients of $q$,
  and
  \begin{equation*}
    \| [p_{ij}(T)]\| = \|q(T \otimes I)\|.
  \end{equation*}
  This shows that $K_{\cb}(m,n) = K(m,n)$.
\end{remark}

With the help of the last result, we can now get quantitative estimates for $K(m,n)$.
One might build a function $h$ in Proposition \ref{prop:band_limited_abstract} with the help of de la Vall\'ee-Poussin kernels
and use known $L^1$-estimates of de la Vall\'ee-Poussin kernels; see
\cite{Mehta15,Steckin78,SSM19}.
Unfortunately, these estimates do not appear to be completely sufficient for our needs.

\begin{remark}
  \label{rem:basic_bounds}
  Here are some simple ways to get quantitative estimates on $K(m,n)$.
  These are already sufficient for our main applications.
  \begin{enumerate}[wide]
    \item For $m=0$, we may use as in \cite[Corollary 4.17]{Pisier01} the function $h(z) = \sum_{k=0}^n z^k = \frac{1 - z^{n+1}}{1 - z}$.
      This function is a shifted version of the Dirichlet kernel, so $\|h\|_{H^1}$ is comparable to $\log(n+1) + 1$.
      Thus,
      \begin{equation*}
        K(0,n) \lesssim \log(n+1) + 1.
      \end{equation*}
      
    \item For general $n \ge m \ge 0$, we can use the shifted Fej\'er-type kernels $W_j$
      whose Fourier coefficients are the triangular-shaped function supported in $(2^{j-1},2^{j+1})$
      with peak at $2^j$; see Appendix \ref{sec:Besov} for more information.
      Let $h = \sum_{j=a}^b W_j$. Then $\widehat{h}(k) = 1$ for $2^a \le k \le 2^b$ (and $\widehat{h}(k) = 1$
      for  $0 \le k \le 2^b$ in case $a=0$) and $\|h\|_{H^1} \le \frac{3}{2} (b-a +1)$.
      (A slightly different function $h$ with these properties was already contructed by Haase;
      see \cite[Lemma A.2]{Haase11}.)
      By Proposition \ref{prop:band_limited_abstract}, this yields
      \begin{equation*}
        K(m,n) \lesssim 1 + \log \Big( \frac{n+1}{m+1} \Big).
      \end{equation*}
    \item Let
  \begin{equation*}
    f(z) = \frac{1}{m+1} \sum_{k=0}^m z^k, \quad g(z) = \sum_{k=0}^n z^k
  \end{equation*}
  and $h = f g$. Then $\widehat{h}(k) = 1$ for $m \le k \le n$, and
  \begin{equation*}
    \|h\|_{H^1}^2 \le \|f\|_{H^2}^2 \|g\|_{H^2}^2
    = \frac{n+1}{m+1}.
  \end{equation*}
  This yields 
  \begin{equation*}
    K(m,n) \le \left( \frac{n+1}{m+1} \right)^{1/2}.
  \end{equation*}
  \end{enumerate}
\end{remark}

Observe that estimate (2) is better than estimate (3) when the ratio $\frac{n+1}{m+1}$ is large, while estimate (3) is better
when $\frac{n+1}{m+1}$ is close to $1$ because of the implied constants in estimate (2).
Even though the estimates in Remark \ref{rem:basic_bounds} are sufficient for our applications, it seems worthwhile to determine
the behavior of $K(m,n)$ more precisely and in particular establish an estimate for $K(m,n)$ that is good in both regimes.
This is done in the following result.

\begin{theorem}
  \label{thm:OVNI_sharp}
  We have
  \begin{equation*}
    \max\Big( 1,\frac{1}{\pi} \log \Big( \frac{n+2}{m+1} \Big) \Big) \le K(m,n) \le \frac{1}{\pi} \log \Big( \frac{n+1}{m+1} \Big)
    + \min \Big( \frac{n+1}{m+1},2 \Big).
  \end{equation*}
\end{theorem}

\begin{proof}
  Upper bound: We use Proposition \ref{prop:band_limited_abstract} and construct a function $h \in H^1$
  with $\widehat{h}(k) = 1$ for $m \le k \le n$ whose norm is dominated by the right-hand side in the statement
  of the theorem.
  The construction below already appears in work of Haase, see the proof of \cite[Lemma A.2]{Haase11}.
  But in order to obtain the stated bound, we need to estimate somewhat more carefully.
  
  Define holomorphic functions on the disc by
  \begin{align*}
    f(z) = \sum_{j=0}^{m-1} \frac{j+1}{m+1} z^j + \sum_{j=m}^\infty z^j
    &= \frac{d}{dz} \Big( \frac{1 - z^{m+1}}{(m+1) (1 - z)} \Big) + \frac{z^m}{1 - z} \\
    &= \frac{1 - z^{m+1}}{(m+1) (1 - z)^2}
  \end{align*}
  and
  \begin{equation*}
    u(z) = \frac{(1 - z^{m+1})^{1/2}}{(m+1)^{1/2} (1 - z)}
  \end{equation*}
  and
  \begin{equation*}
    g(z) = \sum_{j=0}^n \widehat{u}(j) z^j.
  \end{equation*}
  Finally, we set $h = g^2$.
  Since $g$ agrees with $u$ to order $n$ and $u^2 =f$, we see that $h$ agrees with $f$ to order $n$.
  In particular, $\widehat{h}(k) = 1$ for $m \le k \le n$.
  
  It remains to estimate $\|h\|_{H^1}$.
  To this end, notice that
  \begin{equation}
    \label{eqn:h_norm}
    \|h\|_{H^1} = \|g\|_{H^2}^2 = \sum_{j=0}^n |\widehat{u}(j)|^2.
  \end{equation}
  To compute the Taylor coefficients of $u$, we use the binomial series to obtain
  \begin{equation*}
    (m+1)^{\frac{1}{2}} u(z) = \Big( \sum_{k=0}^\infty (-1)^k \binom{\frac{1}{2}}{k} z^{k (m+1)} \Big)
    \Big( \sum_{k=0}^\infty z^k \Big).
  \end{equation*}
  Expanding the product and writing $l = \lfloor{j/(m+1)} \rfloor$,
  we see that for all $j \ge 0$, the Taylor coefficients are given by
  \begin{equation}
    \label{eqn:u_taylor}
    (m+1)^{1/2} \widehat{u}(j) =  \sum_{\nu = 0}^l (-1)^\nu \binom{\frac{1}{2}}{\nu}
    = (-1)^l \binom{-\frac{1}{2}}{l}.
  \end{equation}
  Here, the last identity can be seen by expanding both sides of $(1 - z)^{-1/2} = \frac{(1- z)^{1/2}}{1 -z}$
  into a binomial series and comparing coefficients.
  Expanding the binomial coefficient $\binom{-1/2}{l}$ and rearranging \eqref{eqn:u_taylor} yields
  \begin{equation}
    \label{eqn:Taylor_u}
    \widehat{u}(j) = (m+1)^{-1/2} 4^{-l} \binom{2l}{l}, \quad
    \text{ where } 
  l = \lfloor{j/(m+1)} \rfloor.
  \end{equation}
  In words, the Taylor coefficients of $u$ are given by the sequence on the right,
  and each element is repeated $m+1$ times.

  Now, write $(n+1) = k (m+1) + r$ with natural numbers $k,r$ with $r < m+1$ (hence
  $k = \lfloor{(n+1)/(m+1)} \rfloor)$.
  Thus, from \eqref{eqn:h_norm} and \eqref{eqn:Taylor_u}, we find that
  \begin{equation*}
    \|h\|_{H^1} = \sum_{l=0}^{k-1} \Big(4^{-l} \binom{2l}{l} \Big)^2 + \frac{r}{m+1} \Big( 4^{-k} \binom{2 k}{k} \Big)^2.
  \end{equation*}
  To estimate the central binomial coefficients, we use Stirling's formula in the form
  \begin{equation*}
    \sqrt{2 \pi} n^{n+\frac{1}{2}} e^{-n} e^{\frac{1}{12 n + 1}} \le n!
    \le
    \sqrt{2 \pi} n^{n+\frac{1}{2}} e^{-n} e^{\frac{1}{12 n}},
  \end{equation*}
  see for example \cite{Robbins55}, to obtain the estimate
  \begin{equation*}
    4^{-l} \binom{2l}{l} \le \frac{1}{\sqrt{\pi l}}
  \end{equation*}
  for $l \ge 1$. Thus,
  \begin{equation}
    \label{eqn:h_norm_2}
    \|h\|_{H^1} \le 1 + \sum_{l=1}^{k-1} \frac{1}{\pi l} + \frac{r}{\pi (m+1) k}.
  \end{equation}
  Comparing the sum with an integral, we find that
  \begin{equation*}
    \sum_{l=1}^{k-1} \frac{1}{l} \le \log(k) + \min(k-1,1)
  \end{equation*}
  for $k \ge 1$. Moreover, recalling that $(n+1) = k (m+1) + r$, we see that
  \begin{equation*}
    \frac{r}{(m+1) k} = \frac{\frac{n+1}{m+1} - \lfloor{ \frac{n+1}{m+1}}\rfloor}{ \lfloor{ \frac{n+1}{m+1}} \rfloor}
    \le \min \Big( \frac{n+1}{m+1}-1,1 \Big).
  \end{equation*}
  It therefore follows from \eqref{eqn:h_norm_2} that
  \begin{align*}
    \|h\|_{H^1} \le 1+ \frac{1}{\pi} \log \Big( \frac{n+1}{m+1} \Big) + \frac{2}{\pi} \min \Big( \frac{n+1}{m+1} - 1,1 \Big).
  \end{align*}
  The stated upper bound follows from this inequality.

  Lower bound: The lower bound $K(m,n) \ge 1$ is trivial. Let
  \begin{equation*}
    q(z) = \sum_{k=m}^n \frac{1}{k+1} z^k.
  \end{equation*}
  Then
  \begin{equation*}
    q(1) = \sum_{k=m}^n \frac{1}{k+1} \ge \int_m^{n+1} \frac{1}{t+1} \, dt = \log\big( \frac{n+2}{m+1} \big).
  \end{equation*}
  To estimate $\|q\|_{\Han}$, observe that with respect to the standard bases of $H^2$ and $\overline{H^2}$,
  the Hankel operator $H_q$ is the Hankel matrix corresponding to the sequence $(\overline{\widehat{q}(k)})_{k=0}^\infty$.
  Since $q$ has non-negative Taylor coefficients, it follows that $\|q\|_{\Han}$ is bounded above by the norm of the Hilbert matrix,
  which is equal to $\pi$. Thus, the lower bound follows from Proposition \ref{prop:band_limited_abstract}.
\end{proof}

As a consequence, we obtain a Besov-type functional calculus.
Background on Besov spaces can be found in Appendix \ref{sec:Besov}.
If $f: \mathbb{D} \to B(\mathcal{H})$ is an operator-valued holomorphic function, define
\begin{equation*}
  \|f\|_{B^0_{\infty,1}} = \|f(0)\| + \int_0^1 \|f'_r\|_\infty \, dr.
\end{equation*}
Here, $f'_r(z) = f'(r z)$ and $\|g\|_\infty = \sup_{z \in \mathbb{D}} \|g(z)\|$.

\begin{corollary}
  Let $f: \mathbb{D} \to B(\mathcal{H})$ be an operator-valued analytic function with Taylor series
  \begin{equation*}
    f(z) = \sum_{k=0}^\infty A_k z^k.
  \end{equation*}
  Let $T \in B(\mathcal{H})$ be a contraction that commutes with each $A_k$.
  If $\|f\|_{B^0_\infty,1} < \infty$, then
  \begin{equation*}
    f(T) = \lim_{r \nearrow 1} f( r T) = \lim_{r \nearrow 1} \sum_{n=0}^\infty A_k r^k T^k
  \end{equation*}
  exists, and
  \begin{equation*}
    \|f(T)\| \lesssim \|f\|_{B^0_{\infty,1}}.
  \end{equation*}
\end{corollary}

\begin{proof}
  Suppose initially that $f$ is holomorphic in a neighborhood of $\overline{\mathbb{D}}$, so that the
  sum $\sum_{k=0}^\infty A_k T^k$ converges. We use the Fej\'er-type kernels $W_n$,
  see the appendix. Let $p_n = f * W_n$.
  Since $\sum_{n=0}^\infty \widehat{W}_n(k) = 1$ for all $k$ and since $f$ is holomorphic
  in a neighborhood of $\overline{\mathbb{D}}$, we may interchange the order of summation to find that
  \begin{equation*}
    \sum_{n=0}^\infty p_n(T) = \sum_{n=0}^\infty \sum_{k=0}^\infty A_k \widehat{W}_n(k) T^k
    = \sum_{k=0}^\infty A_k T^k = f(T).
  \end{equation*}
  For $n \ge 1$, the polynomial $p_n$ is $(2^{n-1},2^{n+1})$-band-limited,
  so by Theorem \ref{thm:OVNI_sharp} (or simply Remark \ref{rem:basic_bounds} (2)),
  \begin{equation*}
    \|p_n(T)\| \le K(2^{n-1},2^{n+1}) \|p_n\|_\infty \lesssim \|p_n\|_\infty.
  \end{equation*}
  Such an estimate also holds for $n=0$.
  It follows that
  \begin{equation*}
    \|f(T)\| \le \sum_{n=0}^\infty \|p_n(T)\|
    \lesssim \sum_{n=0}^\infty \|f * W_n\|_\infty.
  \end{equation*}
  By Proposition \ref{prop:Besov_equiv}, the last expression is comparable to $\|f\|_{B^0_{\infty,1}}$.

  If $f$ merely satisfies $\|f\|_{B^0_{\infty,1}} < \infty$, then by definition and the dominated convergence theorem, $\|f - f_r\|_{B^0_{\infty,1}} \to 0$
  as $r \to 1$. From this and the first paragraph, it easily follows that the net $(f_r(T))$ is Cauchy in $B(\mathcal{H})$,
  hence $\lim_{r \to 1} f_r(T)$ exists in $B(\mathcal{H})$ and satisfies the desired norm estimate.
\end{proof}

\section{von Neumann's inequality on the polydisc}
\label{sec:polydisc}

In this section, we use the results on von Neumann's inequality with operator coefficients to study
von Neumann's inequality for commuting contractions. The basic idea is very simple:
we plug in operators successively and use the inequality with operator coefficients in each step.
This approach already appeared in the work of Daher; see \cite[Chapter 4]{Pisier01}.

As a first application, we establish an upper bound for the Schur--Agler norms of polynomials
that is polylogarithmic in the degree of the polynomial.
Supremum norms of $d$-variable functions are understood to be taken over the polydisc $\mathbb{D}^d$.

\begin{proposition}
  \label{prop:SA_log_bound}
  If $d \ge 2$ and if $p \in \mathbb{C}[z_1,\ldots,z_d]$ is a polynomial of degree $n \ge 1$, then
  \begin{equation*}
    \|p\|_{\SA} \lesssim_d (\log(n+1))^{d-2} \|p\|_\infty.
  \end{equation*}
\end{proposition}

\begin{proof}
  The proof is by induction on $d$. If $d=2$, then And\^o's theorem shows that $\|p\|_{\SA} = \|p\|_\infty$.
  Let $d \ge 3$, and suppose that the result has been shown for $d-1$.
  We write
  \begin{equation*}
    p(z) = \sum_{k=0}^n p_k(z_1,\ldots,z_{d-1}) z_d^k,
  \end{equation*}
  for polynomials $p_k \in \mathbb{C}[z_1,\ldots,z_{d-1}]$ of degree at most $n$.
  Let $T$ be a $d$-tuple of commuting contractions.
  By the inductive hypothesis,
  we have
  \begin{equation*}
    \|p(T_1,\ldots,T_{d-1},z)\| \lesssim_d (\log(n+1))^{d-3} \|p\|_\infty
  \end{equation*}
  for all $z \in \mathbb{D}$, so by Theorem \ref{thm:OVNI_sharp} (or simply Remark \ref{rem:basic_bounds} (1)),
  we have
  \begin{equation*}
    \|p(T)\| \le K(0,n) \sup_{z \in \mathbb{D}} \|p(T_1,\ldots,T_{d-1},z)\|
    \lesssim_d (\log(n+1))^{d-2} \|p\|_\infty. \qedhere
  \end{equation*}
\end{proof}

A similar technique as in the last proof can be used to establish Theorem \ref{thm:main3},
which we restate here.

\begin{theorem}
  \label{thm:SA_sup_asympt}
  Let $d \ge 3$ and let $p \in \mathbb{C}[z_1,\ldots,z_d]$. Then
  \begin{equation*}
    \lim_{m \to \infty} \|(z_3 \cdot \ldots \cdot z_d)^m p\|_{\SA} = \|p\|_\infty.
  \end{equation*}
\end{theorem}

\begin{proof}
  It is clear that
  \begin{equation*}
    \|(z_3 \cdot \ldots \cdot z_d)^m p\|_{\SA} \ge \|(z_3 \cdot \ldots \cdot z_d)^m p\|_\infty = \|p\|_\infty
  \end{equation*}
  for all $m \in \mathbb{N}$.

  To establish the converse direction, we prove that for all $d \ge 2$
  and all polynomials $p \in \mathbb{C}[z_1,\ldots,z_d]$ of degree at most $n$, the estimate
  \begin{equation}
    \label{eqn:SA_sup_asympt}
    \|(z_3 \cdot \ldots \cdot z_d)^m p \|_{\SA} \le K(m,m+n)^{d-2} \|p\|_\infty
  \end{equation}
  holds. Here, the (empty) product on the left is understood as $1$ if $d=2$.
  Assuming this estimate for a moment,
  we obtain the remaining inequality from
  Theorem \ref{thm:OVNI_sharp} (or simply Part (3) of Remark \ref{rem:basic_bounds}), which shows that
  \begin{equation*}
    \limsup_{m \to \infty} K(m,m+n) \le 1.
  \end{equation*}

  The proof of \eqref{eqn:SA_sup_asympt} is similar to that of Proposition \ref{prop:SA_log_bound} and proceeds by induction on $d$.
  The case $d=2$ follows from And\^o's theorem.
  Let $d \ge 3$ and suppose that \eqref{eqn:SA_sup_asympt} has been shown for $d-1$. Let $p$ be a polynomial in $d$ variables
  of degree at most $n$, which we write as
  \begin{equation*}
    p(z) = \sum_{k=0}^n p_k(z_1,\ldots,z_{d-1}) z_d^k,
  \end{equation*}
  where each $p_k \in \mathbb{C}[z_1,\ldots,z_{d-1}]$ has degree at most $n$.
  Let $T$ be a $d$-tuple of commuting contractions.
  Then,
  \begin{align*}
    &\| ( (z_3 \cdot \ldots \cdot z_d)^m p)(T)\| \\ = \,
    &\Big\| (T_3 \cdot \ldots \cdot T_{d-1})^m \sum_{k=0}^n  p_k(T_1,\ldots,T_{d-1}) T_d^{k+m} \Big\| \\
                                                \le
    \, &K(m,m+n) \sup_{z_d \in \mathbb{D}}
    \Big\| (T_3 \cdot \ldots \cdot T_{d-1})^m \sum_{k=0}^n p_k(T_1,\ldots,T_{d-1}) z_d^{k+m} \Big\|.
  \end{align*}
  For each $z_d \in \mathbb{D}$, the inductive hypothesis implies that
  \begin{align*}
    &\Big\| (T_3 \cdot \ldots \cdot T_{d-1})^m \sum_{k=0}^n  p_k(T_1,\ldots,T_{d-1}) z_d^{k+m} \Big\| \\
    \le \, &K(m,m+n)^{d-3} \sup_{z' \in \mathbb{D}^{d-1}} | z_d^m p(z',z_d)| \\
    \le \,  &K(m,m+n)^{d-3} \|p\|_\infty.
  \end{align*}
  Combining both inequalities yields \eqref{eqn:SA_sup_asympt} for $d$.
\end{proof}

If $p \in \mathbb{C}[z_1,\ldots,z_d]$ is a polynomial, say
\begin{equation*}
  p(z) = \sum_{\alpha} \widehat{p}(\alpha) z^\alpha,
\end{equation*}
then we say that $p$ is $(m,n)$-band-limited if $\widehat{p}(\alpha) = 0$ whenever $|\alpha| < m$
or $|\alpha| > n$.
We say that $p$ is $(m,n)$-band-limited with respect to $z_j$
if $\widehat{p}(\alpha) = 0$ whenever $\alpha_j < m$ or $\alpha_j > n$.
In other words, if we fix all variables but the $j$th one, then $p$ is $(m,n)$-band-limited
as a polynomial in $z_j$.

The following splitting lemma is crucial for establishing the upper bound of the Schur--Agler norm of homogeneous polynomials.

\begin{lemma}
  \label{lem:splitting}
  Let $p \in \mathbb{C}[z_1,\ldots,z_d]$ be an $(m,n)$-band-limited polynomial.
  Then there exist polynomials $p_1,\ldots,p_d$ such that
  \begin{enumerate}
    \item $p = \sum_{j=1}^d p_j$,
    \item each $p_j$ is $(\lfloor \frac{m}{2 d} \rfloor,n)$-band-limited with respect to $z_j$, and
    \item $\|p_j\|_\infty \lesssim_d \|p\|_\infty$ for each $j$.
  \end{enumerate}
\end{lemma}

\begin{proof}
  We first informally describe the  basic idea. If $\widehat{p}(\alpha) \neq 0$, then $\alpha_j \ge \frac{m}{d}$ for
  some $j$, as $p$ is $(m,n)$-band-limited. So we should assign the monomial $\widehat{p}(\alpha) z^\alpha$
  to the polynomial $p_j$. In this way, one obtains a splitting that satisfies
  Conditions (1) and (2), even with $\frac{m}{d}$ in place of $\frac{m}{2 d}$. However, to maintain supremum norm control, we need to smoothen the cut-off.
  This will be achieved with the help of de la Vall\'ee-Poussin kernels.

  We now come to the actual proof.
  By replacing $m$ with $\lfloor \frac{m}{2 d} \rfloor 2 d$, we may assume that $\frac{m}{2 d}$ is an integer.
  For an integer $k \ge 2$, let $V_k$ be the real-valued trigonometric polynomial of one variable
  whose non-negative Fourier coefficients are the trapezoid-shaped function supported in $(\frac{m}{2d},kn)$
  that is identically one on $[\frac{m}{d},n]$ and affine on $[\frac{m}{2d},\frac{m}{d}]$
  and on $[n,kn]$.

  If $q \in \mathbb{C}[z_1,\ldots,z_d]$ is any polynomial, we write
  \begin{align*}
    (q *_j V_k)(z) &= \int_{\mathbb{T}} q(z_1,\ldots,z_{j-1}, z_j \overline{w}, z_{j+1},\ldots,z_d) V_k(w) \frac{|d w|}{2 \pi} \\
    &= \sum_{\alpha} \widehat{q}(\alpha) \widehat{V}_k(\alpha_j) z^\alpha.
  \end{align*}
  If $q$ has degree at most $n$, then $q *_j V_k$ is $(\frac{m}{2d},n)$-band-limited with respect to $z_j$
  and independent of $k$.
  Since $\|V_k\|_{L^1} \le 3 + \frac{k+1}{k-1}$ (see Lemma \ref{lem:dlvp_estimate} in the appendix),
  we have $\|q *_j V_k\|_\infty \le 4 \|q\|_\infty$.

  Recursively, we define $p_1 = p *_1 V$ and
  \begin{equation*}
    p_j = (p - p_{1} - \ldots - p_{j-1})*_j V_k
  \end{equation*}
  for $j=2,3,\ldots d$. Properties (2) and (3) are then clear.
  To show Property (1), let
  $q = p - p_1 - \ldots - p_{d-1}$. A simple induction argument shows that
  \begin{equation*}
    \widehat{q}(\alpha) = \widehat{p}(\alpha) \prod_{j=1}^{d-1} (1 - \widehat{V_k}(\alpha_j))
  \end{equation*}
  for all multi-indices $\alpha$.
  Thus, if $\alpha$ is a multi-index with $ \alpha_k \ge \frac{m}{d}$ for some $1 \le k \le d-1$,
  then $\widehat{q}(\alpha) = 0$.
  On the other hand, $q$ is $(m,n)$-band limited, so it follows
  that $\widehat{q}(\alpha) = 0$ if $\alpha_d < \frac{m}{d}$.
  Consequently,
  \begin{equation*}
    p_d = q *_d V_k = q,
  \end{equation*}
  which gives (1).
\end{proof}

The following result will imply both results mentioned in the introduction fairly easily.

\begin{theorem}
  \label{thm:SA_band_limited_estimate}
  Let $d \ge 3$ and let $p \in \mathbb{C}[z_1,\ldots,z_d]$ be $(m,n)$-band limited
  with $n \ge 1$.
  Then
  \begin{equation*}
    \|p\|_{\SA} \lesssim_d 
    \Big( \log \Big( \frac{n+1}{m+1} \Big) + 1 \Big) (\log(n+1))^{d-3} \|p\|_\infty.
  \end{equation*}
\end{theorem}

\begin{proof}
  By Lemma \ref{lem:splitting}, we may split $p = \sum_{j=1}^d p_j$, where each $p_j$ is $(\lfloor \frac{m}{2d} \rfloor,n)$-band-limited with respect to $z_j$ and $\|p_j\|_\infty \lesssim_d \|p\|_\infty$.
  Write
  \begin{equation*}
    p_j(z) = \sum_{k= \lfloor \frac{m}{2d} \rfloor}^n q_{k j}(z_1,\ldots,z_{j-1},z_{j+1},\ldots,z_d) z_j^k,
  \end{equation*}
  where each $q_{kj}$ is a polynomial of degree at most $n$.
  Let $T$ be a $d$-tuple of commuting contractions.
  By Proposition \ref{prop:SA_log_bound}, we have
  \begin{equation*}
    \|p_j(T_1,\ldots,T_{j-1},z,T_{j+1},\ldots,T_d)\| \lesssim_d (\log(n+1))^{d-3} \|p_j\|_\infty
  \end{equation*}
  for all $z \in \mathbb{D}$, so by Theorem \ref{thm:OVNI_sharp} (or simply Remark \ref{rem:basic_bounds} (2)), we have
  \begin{align*}
    \|p_j(T)\| &\lesssim_d K \big( \lfloor \frac{m}{2d} \rfloor,n \big) (\log(n+1))^{d-3} \|p_j\|_\infty \\
               &\lesssim_d \Big( \log \Big( \frac{n+1}{\lfloor \frac{m}{2d} \rfloor+1} \Big) + 1 \Big) (\log(n+1))^{d-3} \|p_j\|_\infty \\
               &\lesssim_d \Big( \log \Big( \frac{n+1}{m+1} \Big) + 1 \Big) (\log(n+1))^{d-3} \|p_j\|_\infty.
  \end{align*}
  Recalling that $\|p_j\|_\infty \lesssim_d \|p\|_\infty$, and that $p = \sum_{j=1}^d p_j$,
  we find that
  \begin{equation*}
    \|p(T)\| \lesssim_d
    \Big( \log \Big( \frac{n+1}{m+1} \Big) + 1 \Big) (\log(n+1))^{d-3} \|p\|_\infty,
  \end{equation*}
  as desired.
\end{proof}

The desired upper bound for $C(d,n)$ follows immediately by taking $m=n$ above.
\begin{corollary}
  \label{cor:C(d,n)_upper_bound}
  Let $d \ge 3$. Then for all $n \ge 1$,
  \begin{equation*}
    C(d,n) \lesssim_d (\log(n+1))^{d-3}.
  \end{equation*}
  In particular, $\sup_{n} C(3,n) < \infty$.
\end{corollary}

\begin{remark}
  \label{rem:explicit}
  (a)
  The proof of Corollary \ref{cor:C(d,n)_upper_bound} does yield a crude explicit estimate for $C(3,n)$.
  Let $p \in \mathbb{C}[z_1,z_2,z_3]$ be a homogeneous polynomial of degree $n$ with $\|p\|_\infty \le 1$.
  The proof of Lemma \ref{lem:splitting} shows that the splitting $p = p_1 + p_2 + p_3$
  obeys the bounds $\|p_1\|_\infty \le 4, \|p_2\|_\infty \le (4 + 1) 4 = 20$ and $\|p_3\|_\infty \le 1 + 4 + 20 = 25 $.
  Moreover, by Remark \ref{rem:basic_bounds} (3), we have $K(\lfloor \frac{n}{6} \rfloor, n) \le \sqrt{6}$.
  Thus, the proof of Theorem \ref{thm:SA_band_limited_estimate} yields $C(3,n) \le \sqrt{6} (4 + 20+25) = 49 \sqrt{6} \le 121$.
  Undoubtedly, this estimate can be significantly improved, but no attempt was made to do so.

  (b) Once again, one can define a completely bounded version $C_{\cb}(d,n)$ of the constant $C(d,n)$
  by considering matrices of homogeneous polynomials in place of ordinary homogeneous polynomials.
  The arguments above also yield, for each $d \ge 3$ and $n \ge 1$, the estimate
  \begin{equation*}
    C_{\cb}(d,n) \lesssim_d (\log(n+1))^{d-3}.
  \end{equation*}
  Indeed, by Remark \ref{rem:power_bounded_and_cb} (b), the upper bound in Theorem \ref{thm:OVNI_sharp}
  also holds in the completely bounded setting, and Proposition \ref{prop:SA_log_bound} and Lemma \ref{lem:splitting}
  extend to matrices of polynomials with essentially the same proofs.
  Hence the argument in Theorem \ref{thm:SA_band_limited_estimate} extends as well.
\end{remark}

The estimate in Theorem \ref{thm:SA_band_limited_estimate} can be converted into a Besov norm upper bound for the Schur--Agler norm.
If $d \ge 3$ and $f: \mathbb{D}^d \to \mathbb{C}$ is holomorphic,
let $(R f)(z) = \sum_{j=1}^d z_j \frac{\partial f}{\partial z_j}$ be the radial derivative. Let us define
\begin{equation*}
  \|f\|_d = |f(0)| + \int_{0}^1 \|(R f)_r\|_\infty \Big(\log\Big(\frac{1}{1-r} \Big) \Big)^{d - 3} \, dr.
\end{equation*}
If $f: \mathbb{D}^d \to \mathbb{C}$ is holomorphic, we define the Schur--Agler norm by $\|f\|_{\SA} = \sup\{\|f(T)\|\}$,
where the supremum is taken over all commuting $d$-tuples of strict contractions,
and $f(T)$ is (for instance) defined with the help of power series.
We say that $f$ belongs to the Schur--Agler algebra if $\|f\|_{\SA} < \infty$.

\begin{corollary}
  \label{cor:Besov_upper_bound}
  Let $d \ge 3$. If $f: \mathbb{D}^d \to \mathbb{C}$ is holomorphic and $\|f\|_d < \infty$,
  then $f$ belongs to the Schur-Agler algebra and
  \begin{equation*}
    \|f\|_{\SA} \lesssim_d \|f\|_d.
  \end{equation*}
\end{corollary}

\begin{proof}
  We use the decomposition $f = \sum_{n=0}^\infty f * W_n$, which converges uniformly on compact subsets of
  ${\mathbb{D}}^d$; see the appendix.
  Let $T$ be a commuting tuple of strict contractions.
  Note that $f * W_n$ is $(2^{n-1},2^{n+1})$-band-limited for $n \ge 1$, so by Theorem \ref{thm:SA_band_limited_estimate}, we find that
  \begin{equation*}
    \|f(T)\| \le \sum_{n=0}^\infty \|(f * W_n)(T)\|
    \lesssim_d \sum_{n=0}^\infty (n+1)^{d-3} \|(f* W_n)\|_\infty.
  \end{equation*}
  By Corollary \ref{cor:Besov_polydisc}, the right-hand side is comparable to $\|f\|_d$,
  which gives the result.
\end{proof}

The bounds in von Neumann's inequality with operator coefficients used in the proof of Theorem \ref{thm:SA_band_limited_estimate} were essentially sharp; see Theorem \ref{thm:OVNI_sharp}.
One might therefore try to establish sharpness of Theorem \ref{thm:SA_band_limited_estimate} in a similar way.
Recall that the lower bound in von Neumann's inequality with operator coefficients
was achieved by Foguel--Hankel operators
of the form
\begin{equation*}
  \begin{bmatrix}
    M_\zeta & 0 \\ 0 & M_{\overline{\zeta}}^*
  \end{bmatrix} \quad \text{ and } \quad
  \begin{bmatrix}
    0 & 0 \\ H & 0
  \end{bmatrix},
\end{equation*}
where $M_\zeta$ is the unilateral shift on $H^2$ and $H: H^2 \to \overline{H^2}$ is a Hankel operator, i.e.\ $H M_\zeta = M_{\overline{\zeta}}^* H$.
It therefore seems natural to try to construct counterexamples for von Neumann's inequality in $d$-variables
using operators of this type. We will show that some natural operator tuples
built in this way in fact satisfy von Neumann's inequality with constant $1$.
We require the following standard lemma.

\begin{lemma}
  \label{lem:2_by_2_contraction}
  Let $V \in B(\mathcal{H})$ be an isometry, let $W \in B(\mathcal{K})$ be a co-isometry,
  let $H \in B(\mathcal{H},\mathcal{K})$ and let $r \in [0,\infty)$. Then the operator
  \begin{equation*}
    \begin{bmatrix}
      r V & 0 \\
      H & r W
    \end{bmatrix} \in B(\mathcal{H} \oplus \mathcal{K})
  \end{equation*}
  is a contraction if and only if $r^2 + \|H\| \le 1$.
\end{lemma}

\begin{proof}
  Let $T$ denote the operator in the statement. Since $V$ is an isometry, it is clear that $r \le 1$ is necessary.
  If $r=1$, then $\|T\| \le 1$ if and only if $H = 0$. Thus, we may assume that $r \in (0,1)$.
  We let $I$ denote the identity operator on $\mathcal{H} \oplus \mathcal{K}$.
  Observe that $T$ is a contraction if and only if $I - T^* T \ge 0$.
  Now,
  \begin{equation*}
    I - T^* T
    =
    \begin{bmatrix}
      1 - r^2 - H^* H & - r H^* W \\
       - r W^* H & 1 -  r^2 W^* W
    \end{bmatrix}.
  \end{equation*}
  Taking Schur complements, we find that this operator is positive if and only if
  \begin{equation*}
    1 - r^2 - H^* H - r^2 H^* W (1 - r^2 W^* W)^{-1} W^* H \ge 0.
  \end{equation*}
  Since $W W^* = 1$, the left-hand side equals
  \begin{equation*}
    1 - r^2 - (1 + r^2 (1-r^2)^{-1}) H^* H = (1-r^2)^{-1} ( (1-r^2)^2 - H^* H).
  \end{equation*}
  Thus, $\|T\| \le 1$ if and only if $\|H\| \le 1 - r^2$.
\end{proof}

We can now show the announced result about tuples of Foguel--Hankel type.
Recall that the Hankel operator $H_q: H^2 \to \overline{H^2}$ is defined
by $\langle H_q f,g  \rangle = \langle f g,q \rangle$ for polynomials $p,q$.
Equivalently, $H_q f = P_{\overline{H^2}} ( \overline{q} f)$, where the product is taken in $L^2$.
We let $\Han \subset H^2$ denote the space of symbols of bounded Hankel operators.

\begin{proposition}
  \label{prop:Hankel_constant_1}
    Let $q_1,\ldots,q_d \in \Han$ and let $r_1,\ldots,r_d \in [0,1]$. For $j=1,\ldots,d$, let
      \begin{equation*}
        T_j =
        \begin{bmatrix}
          r_j M_{\zeta} & 0 \\ H_{q_j} & r_j M_{\overline{\zeta}}^*
        \end{bmatrix},
      \end{equation*}
      and assume that each $T_j$ is a contraction.
      Then the $T_j$ commute and for every $p \in \mathbb{C}[z_1,\ldots,z_d]$, we have
      \begin{equation*}
        \|p(T_1,\ldots,T_d)\| \le \|p\|_\infty.
      \end{equation*}
\end{proposition}

\begin{proof}
  The $T_j$ commute because of the relation $H_q M_\zeta = M_{\overline{\zeta}}^* H_q$
  for every $q \in \Han$.
  Let $U: L^2 \to L^2, (U f)(\zeta) = \zeta f(\zeta)$ denote the bilateral shift.
  Lemma \ref{lem:2_by_2_contraction} shows
  that $\|H_{q_j}\| \le 1 - r_j^2$ for each $j$.
  By Nehari's theorem (\cite{Nehari57}, see also \cite[Theorem 1.1]{Peller03}),
  there exist $h_j \in L^\infty(\mathbb{T})$ such that $\|h_j\|_\infty \le 1 - r_j^2$
  and such that $H_{q_j} f = P_{\overline{H^2}} (h_j f)$ for all $f \in H^2$.
  Let $U_j: L^2 \to L^2, U_j(g) = h_j g$.
  Then
  \begin{equation*}
    N_j =
    \begin{bmatrix}
      r_j U & 0 \\ U_j & r_j U
    \end{bmatrix} \in B(L^2 \oplus L^2)
  \end{equation*}
  are commuting contractions by Lemma \ref{lem:2_by_2_contraction}. They dilate $T_1,\ldots,T_d$
  in the sense that
  \begin{equation*}
    p(T_1,\ldots,T_d) =
    P_{H^2 \oplus \overline{H^2}} p(N_1,\ldots,N_d) \big|_{H^2 \oplus \overline{H^2}}
  \end{equation*}
  for every polynomial $p \in \mathbb{C}[z_1,\ldots,z_d]$.
  This can for instance be seen
  by noting that the equality holds for $p = z_j$ and that $H^2 \oplus \overline{H^2} = (H^2 \oplus L^2) \ominus (0 \oplus \overline{H^2}^\bot)$ is semi-invariant under $N_1,\ldots,N_j$.
  The entries of $N_j$ are multiplication operators on $L^2$, so
  \begin{align*}
    \|p(T_1,\ldots,T_d)\| &\le \|p(N_1,\ldots,N_d)\| \\
    &\le \sup_{\zeta \in \mathbb{T}}
    \Big\|
    p \Big(
    \begin{bmatrix}
      r_1 \zeta & 0 \\ h_1(\zeta) & r_1 \zeta
    \end{bmatrix},
    \ldots,
    \begin{bmatrix}
      r_d \zeta & 0 \\ h_d(\zeta) & r_d \zeta
  \end{bmatrix} \Big) \Big\|.
  \end{align*}
  Since von Neumann's inequality holds for commuting $2 \times 2$ contractions (see \cite{Drury83} or \cite{Holbrook92}),
  this last quantity is dominated by $\|p\|_\infty$.
\end{proof}

More generally, one might consider Hankel operators on the polydisc.
The following result shows in particular that also in this case,
von Neumann's inequality holds, at least up to a constant.

\begin{proposition}
  \label{prop:Hankel_general_constant}
  Let $V_1,\ldots,V_d \in B(\mathcal{H})$ be commuting isometries
  and let $W_1,\ldots,W_d \in B(\mathcal{K})$ be commuting co-isometries.
  Let $H_1,\ldots,H_d \in B(\mathcal{H},\mathcal{K})$ be operators satisfying $H_i V_j = W_j H_i$
  for $i,j=1,\ldots,d$. Let $r_1, \ldots,r_d \in [0,1]$ and assume that
  the commuting operators
  \begin{equation*}
    T_j =
    \begin{bmatrix}
      r_j V_j & 0 \\ H_j & r_j W_j
    \end{bmatrix} \in B(\mathcal{H} \oplus \mathcal{K})
  \end{equation*}
  are contractions. Then
  \begin{equation*}
    \|p(T_1,\ldots,T_d)\| \le (d+1) \|p\|_\infty
  \end{equation*}
  for all $p \in \mathbb{C}[z_1,\ldots,z_d]$.
\end{proposition}

\begin{proof}
  Let $p \in \mathbb{C}[z_1,\ldots,z_d]$. We claim that
  \begin{equation}
    \label{eqn:Foguel-Hankel}
    p(T_1,\ldots,T_d)
    =
    \begin{bmatrix}
      p(r_1 V_1, \ldots, r_d V_d) & 0 \\
      \sum_{j=1}^d H_j \frac{\partial p}{\partial z_j} (r_1 V_1,\ldots, r_d V_d) & p(r_1 W_1,\ldots, r_d W_d)
    \end{bmatrix}.
  \end{equation}
  Indeed, a simple induction argument using the intertwining relations $H_i V_j = W_j H_i$ shows that this formula holds for all monomials,
  hence it holds for all polynomials by linearity.
  It is a result of It\^o \cite{Ito58}, see also \cite[Theorem 5.1]{Paulsen02}, that
  commuting isometries extend to commuting unitaries, hence
  $\|p(r_1 V_1, \ldots, r_d V_d)\| \le \|p\|_\infty$
  and similarly $\|p(r_1 W_1,\ldots,r_d W_d)\| \le \|p\|_\infty$.
  Applying the same result to $\frac{\partial p}{\partial z_j}$ and
  then using the classical Schwarz--Pick lemma, we find that
  \begin{align*}
    \Big\| \frac{\partial p}{\partial z_j} (r_1 V_1,\ldots,r_d V_d) \Big\|
    &\le \sup \Big\{ \Big| \frac{\partial p}{\partial z_j}(\zeta_1,\ldots,\zeta_d) \Big|: |\zeta_j| \le r_j \Big\} \\
    &\le \frac{1}{1-r_j^2}  \|p\|_\infty,
  \end{align*}
  provided that $r_j < 1$.
  Lemma \ref{lem:2_by_2_contraction} shows that $\|H_j\| \le 1 - r_j^2$ for each $j$; whence
  \begin{equation*}
    \Big\|
      \sum_{j=1}^d H_j \frac{\partial p}{\partial z_j} (r_1 V_1,\ldots, r_d V_d) \Big\| \le d \|p\|_\infty.
  \end{equation*}
  Since the norms of the diagonal entries of \eqref{eqn:Foguel-Hankel} are bounded by $\|p\|_\infty$,
  the desired estimate follows from the triangle inequality.
\end{proof}

\appendix

\section{Besov spaces}
\label{sec:Besov}

In this appendix, we collect a few results about Besov spaces of analytic functions.
Results of this type are in principle well known, see for instance \cite[Section 3.1]{Nikolski05},
\cite[Section 2]{Peller82} for analytic Besov spaces and \cite{Peetre76,Triebel78} for more general Besov spaces.
However, we require vector-valued and weighted versions of the standard results,
for which we do not have a reference. Thus, we provide the proofs.

We will make use of several integral kernels on $\mathbb{T}$.
For integers $n \ge 1$, the Fej\'er kernel $F_n$ is the real-valued trigonometric polynomial whose Fourier coefficients are the triangular-shaped function supported in $(-n,n)$ that is $1$ at $0$ and 
  affine on $[-n,0]$ and on $[0,n]$.
  Explicitly
  \begin{equation*}
    F_n(z) = \sum_{|j| \le n} \Big( 1 - \frac{|j|}{n} \Big) z^j.
  \end{equation*}
  Next, for integers $0 \le k < l \le m < n$, the de la Vall\'ee-Poussin-type kernel $V_{k,l,m,n}$ is defined to be the real-valued trigonometric polynomial
whose non-negative Fourier coefficients are the trapezoid-shaped function
supported in $(k,n)$ that is identically $1$ on $[l,m]$ and affine on $[k,l]$ and
$[m,n]$.
Explicitly,
\begin{equation*}
  V_{k,l,m,n}(z) = \sum_{k \le |j| < l} \frac{|j|-k}{l-k} z^j + \sum_{l \le |j| \le m} z^j + \sum_{m < |j| \le n} \frac{n-|j|}{n-m} z^j.
\end{equation*}
Finally, we will also use closely related holomorphic kernels $W_n$, defined by demanding
the Fourier coefficients of $W_n$ are the triangular-shaped function supported in $(2^{n-1}, 2^{n+1})$ that takes the value $1$ at $2^n$ and is affine on $[2^{n-1},2^n]$ and $[2^n, 2^{n+1}]$.
Explicitly,
\begin{equation*}
  W_n(z) = \sum_{j=2^{n-1}}^{2^n} \Big( \frac{j}{2^{n-1}} - 1 \Big) z^j
  + \sum_{j=2^{n} + 1}^{2^{n+1}} \Big(2 - \frac{j}{2^n} \Big) z^j.
\end{equation*}
We also set $W_0(z) = 1 + z$.

We recall the following standard and well-known estimates for these kernels.
\begin{lemma}
  \label{lem:dlvp_estimate}
  For all integers $0 \le k < l \le m < n$, the following estimate holds:
  \begin{equation*}
    \|V_{k,l,m,n}\|_{L^1} \le \frac{n+m}{n-m} + \frac{l+k}{l-k}.
  \end{equation*}
  Moreover, for all integers $n \ge 0$,
  \begin{equation*}
    \|W_n\|_{L^1} \le \frac{3}{2}.
  \end{equation*}
\end{lemma}

\begin{proof}
  It is well known that the Fej\'er kernel satisfies the estimate $\|F_n\|_{L^1} \le 1$.
  Let
  \begin{equation*}
    G_{m,n} = \frac{n}{n-m} F_n - \frac{m}{n-m} F_m.
  \end{equation*}
  Then the Fourier coefficients of $G_{m,n}$ are the trapezoid-shaped function supported in $(-n,n)$ that is identically $1$ on $[-m,m]$ and affine on $[-n,-m]$ and $[m,n]$. Thus,
  \begin{equation*}
    V_{k,l,m,n} = G_{m,n} - G_{k,l},
  \end{equation*}
  so the estimate for $V_{k,l,m,n}$ follows from the triangle inequality.

  Similarly, 
  $W_n = z^{2^n} F_{2^{n-1}} + \frac{1}{2} z^{3 \cdot 2^{n-1}} F_{2^{n-1}}$
  for $n \ge 1$, which yields the estimate for $W_n$; the case $n=0$
  is an elementary computation.
\end{proof}

Let $X$ be a Banach space. Given a holomorphic function $f: \mathbb{D} \to X$,
we write $\|f\|_\infty = \sup_{ z \in \mathbb{D}} \|f(z)\|_X$, $f_r(z) = f(r z)$ and $f'_r(z) = f'(r z)$.
We also write the Taylor series of $f$ as
\begin{equation*}
  f(z) = \sum_{n=0}^\infty \widehat{f}(n) z^n,
\end{equation*}
where $\widehat{f}(n) \in X$.

We require the following standard inequalities,
which include versions of Bernstein's inequality; see for instance \cite[Section I.8]{Katznelson04}
and \cite{QZ19}.
\begin{lemma}
  \label{lem:Bernstein}
  Let $X$ be a Banach space and let $f: \mathbb{D} \to X$ be holomorphic.
  Let $n \in \mathbb{N}$ and $0 < r < 1$.
  \begin{enumerate}[label=\normalfont{(\alph*)}]
    \item If $\supp \widehat{f} \subset [0,n]$, then $\|f\|_\infty \le r^{-n} \|f_r\|_\infty$
      and $\|f'\|_\infty \lesssim n \|f\|_\infty$.
    \item If $\supp \widehat{f} \subset [n,\infty]$, then $\|f_r\|_\infty \le r^n \|f\|_\infty$
      and $n \|f\|_\infty \lesssim \|f'\|_\infty$.
  \end{enumerate}
\end{lemma}

\begin{proof}
  By composing $f$ with continuous linear functionals on $X$, it suffices to consider the case $X = \mathbb{C}$.

  (a) The first inequality follows from the maximum principle,
  applied to the function $w \mapsto w^n f(\frac{z}{w})$ on the circle of radius $\frac{1}{r}$.
  The second inequality (with constant $1$) is Bernstein's inequality. With a worse
  constant, it also follows from the first inequality, applied to $f'$, and the Schwarz--Pick lemma by
  choosing $r = 1 - \frac{1}{n}$.

  (b) Write $f = z^n g$. Then $\|f_r\|_\infty \le r^n \|g\|_\infty = r^n \|f\|_\infty$, which is the first inequality.
  The second inequality is a special case of the reverse Bernstein inequality.
  It also can be seen from the first inequality as follows. Let $n \ge 1$, so $f(0) = 0$. We have
  \begin{equation*}
    f(z) = f(0) + \int_{0}^1 f'(r z) z \, dr,
  \end{equation*}
  so
  \begin{equation*}
    \|f\|_\infty \le \int_{0}^1 \|f_r'\|_\infty \, dr \le \|f'\|_\infty \int_{0}^1 r^{n-1} \, dr
    = \frac{1}{n} \|f'\|_\infty. \qedhere
  \end{equation*}
\end{proof}

We also require the following basic asymptotic relation.

\begin{lemma}
  \label{lem:integral_asympt}
  Let $a \ge 0$. Then for all $N \in \mathbb{N}, N \ge 1$,
  \begin{equation*}
    N \int_0^1 r^N \Big( \log \Big( \frac{1}{1-r} \Big) \Big)^a \, dr \simeq_a (\log(N+1) )^a.
  \end{equation*}
\end{lemma}

\begin{proof}
  Lower bound: By monotonicity, we have
  \begin{align*}
    N \int_0^1 r^N \Big(\log \Big( \frac{1}{1-r} \Big) \Big)^a \,d r
    &\ge N \int_{1 - \frac{1}{N}}^1 r^N \Big(\log \Big( \frac{1}{1-r} \Big) \Big)^a \, dr \\
    &\ge \Big(1 - \frac{1}{N} \Big)^N (\log(N))^a \\
    &\gtrsim (\log(N))^a.
  \end{align*}
  Upper bound: We break up the domain of integration into two intervals, namely $[0,1-\frac{1}{N}]$
  and $[1-\frac{1}{N},1]$. For the first integral, we estimate
  \begin{equation*}
    N \int_0^{1- \frac{1}{N}} r^N \Big(\log \Big( \frac{1}{1-r} \Big)\Big)^a \, dr
    \le N \int_{0}^{1-\frac{1}{N}} r^N (\log(N))^a \, dr \le (\log(N))^a.
  \end{equation*}
  For the second integral, we use the substitution $s = N (r-1) + 1$ to compute
  \begin{align*}
    N \int_{1-\frac{1}{N}}^1 r^N \Big(\log \Big(\frac{1}{1-r} \Big)\Big)^a \, dr
    &\le
    N \int_{1-\frac{1}{N}}^1 \Big(\log \Big(\frac{1}{1-r} \Big)\Big)^a \, dr \\
    &= \int_0^1 \Big( \log \Big(\frac{N}{1-s} \Big)\Big)^a \, ds \\
    &\lesssim_a (\log(N))^a + \int_0^1 \Big(\log\Big( \frac{1}{1-s} \Big)\Big)^a \,ds.
  \end{align*}
  The second summand is a constant only depending on $a$, which gives the upper bound.
\end{proof}

If $f: \mathbb{D} \to X$ is holomorphic with Taylor series
$f(z) = \sum_{k=0}^\infty \widehat{f}(k) z^n$, where $\widehat{f}(k) \in X$, define
\begin{equation*}
  (f*W_n)(z) = \sum_{k=0}^\infty \widehat{f}(k) \widehat{W}_n(k) z^k,
\end{equation*}
which is in fact a finite sum. Since $\sum_{n=0}^\infty \widehat{W}_n(k) = 1$ for all $k$,
it is easy to check that
\begin{equation*}
  \sum_{n=0}^\infty f * W_n = f,
\end{equation*}
where the convergence is uniform on compact subsets of $\mathbb{D}$.
We also write $(R f)(z) = z f'(z)$.

We can now establish a dyadic description of a Besov-type norm.

\begin{proposition}
  \label{prop:Besov_equiv}
  Let $a \ge 0$ and $f: \mathbb{D} \to X$ be holomorphic.
  Then
  \begin{align*}
    & \|f(0)\| + \int_0^1 \|f'_r\|_\infty \Big(\log \Big( \frac{1}{1-r} \Big)\Big)^a \, d r \\
    \simeq
    \, &\|f(0)\| + \int_0^1 \|(R f)_r\|_\infty \Big(\log \Big( \frac{1}{1-r} \Big)\Big)^a \, d r \\
    \simeq_a &\sum_{n=0}^\infty (n+1)^a \|f * W_n\|_\infty.
  \end{align*}
\end{proposition}

\begin{proof}
  Note that $\|(R f)_r\|_\infty = r \|f'_r\|_\infty$,
  and that $\|f'_r\|_\infty$ is increasing in $r$ by the maximum principle.
  So the integrals over $[0,1]$ are comparable to the respective integrals over $[\frac{1}{2},1]$.
  Hence the first two quantities are comparable.
  
  Next, we obtain an upper bound for the second quantity in terms of the third.
  It is clear that $\|f(0)\| \le \|f * W_0\|_\infty$.
  By monotone convergence, we have
   \begin{equation*}
     \int_0^1 \|(R f)_r\|_\infty \Big(\log \Big(\frac{1}{1-r} \Big)\Big)^a \, dr \le \sum_{n=0}^\infty \int_0^1 \|(R f)_r * W_n\|_\infty \Big(\log \Big(\frac{1}{1-r} \Big)\Big)^a \, dr.
   \end{equation*}
   Applying Lemma \ref{lem:Bernstein} (b) and then (a), we find that for $n \ge 1$,
   \begin{equation*}
     \|(R f)_r * W_n\|_\infty \le r^{2^{n-1}} \|(R f) * W_n\|_\infty
     \lesssim r^{2^{n-1}} 2^{n+1} \|f * W_n\|_\infty.
   \end{equation*}
   Integration and Lemma \ref{lem:integral_asympt} yield
   \begin{align*}
      &\int_{0}^1 \|(R f)_r * W_n\|_\infty \Big(\log\Big( \frac{1}{1-r} \Big)\Big)^a \, dr \\
     \lesssim & \, 2^{n+1} \int_{0}^1 r^{2^{n-1}} \Big(\log \Big(\frac{1}{1-r} \Big)\Big)^a \, dr \|f * W_n\|_\infty \\
     \lesssim_a & \, (n+1)^a \|f * W_n\|_\infty,
   \end{align*}
   which is also valid for $n=0$. Thus,
   \begin{equation*}
     \int_{0}^1 \|(R f)_r\|_\infty \, d r \lesssim \sum_{n=0}^\infty (n+1)^a \|f * W_n\|_\infty.
   \end{equation*}

   For the lower bound, we use the fact that  $\|g * W_n \|_\infty \le \frac{3}{2} \|g\|_\infty$
   as $\|W_n\|_{L^1} \le \frac{3}{2}$ to obtain
   \begin{align*}
     &\int_0^1 \|(R f)_r \|_\infty  \Big(\log \Big( \frac{1}{1-r} \Big)\Big)^a \, dr \\
     \ge
     &\sum_{n=1}^\infty \int_{1-2^{-n}}^{1 - 2^{-n-1}}
     \|(R f)_r\|_\infty  \Big(\log \Big( \frac{1}{1-r} \Big)\Big)^a \, dr \\
    \gtrsim
    &\sum_{n=1}^\infty \int_{1-2^{-n}}^{1 - 2^{-n-1}}
     \|(R f)_r * W_n\|_\infty  \Big(\log \Big( \frac{1}{1-r} \Big)\Big)^a \, dr.
    \end{align*}
   Using Lemma \ref{lem:Bernstein} (a) and then (b), we estimate
   \begin{equation*}
     \|(R f)_r * W_n\|_\infty \ge r^{2^{n+1}} \|(R f)* W_n\|_\infty
     \gtrsim r^{2^{n+1}} 2^{n-1} \|f * W_n\|_\infty
   \end{equation*}
   for $n \ge 1$. It follows that
   \begin{align*}
     &\int_{1-2^{-n}}^{1 - 2^{-n-1}}
     \|(R f)_r * W_n\|_\infty  \Big(\log \Big( \frac{1}{1-r} \Big)\Big)^a \, dr \\
     \gtrsim &\|f * W_n\|_\infty 2^{n-1}
     \int_{1-2^{-n}}^{1 - 2^{-n-1}} r^{2^{n+1}} \Big(\log \Big( \frac{1}{1-r} \Big)\Big)^a \, dr \\
     \gtrsim_a &\|f * W_n\|_\infty  \, n^a.
   \end{align*}
   Here, the last estimate is obtained by bounding below the integrand by a constant.
   Hence,
   \begin{equation*}
     \int_0^1 \|(R f)_r \|_\infty \Big(\log \Big( \frac{1}{1-r} \Big)\Big)^a \, dr \gtrsim_a \sum_{n=1}^\infty (n+1)^a \|f * W_n\|_\infty.
   \end{equation*}
   To deal with the summand for $n=0$,
   we use a simple integration and the maximum principle to see that
   \begin{align*}
     \|f\|_\infty \le \|f(0)\| + \int_0^1 \|f_r'\|_\infty \ dr
     &\lesssim \|f(0)\| + \int_{1-\frac{1}{e}}^1 \|(R f)_r\|_\infty \, dr \\
     &\le \|f(0)\| + \int_0^1 \|(R f)_r\|_\infty \Big(\log \Big(\frac{1}{1-r} \Big)\Big)^a \, dr.
   \end{align*}
   Since $\|f * W_0\|_\infty \lesssim \|f\|_\infty$, the result follows.
\end{proof}

If $f: \mathbb{D}^d \to \mathbb{C}$ is holomorphic, we let $(R f)(z) = \sum_{j=1}^d z_j \frac{\partial f}{\partial z_j}$ be the radial derivative.
We also set
\begin{equation*}
  (f * W_n)(z) = \sum_{\alpha} \widehat{f}(\alpha) W_n(|\alpha|) z^\alpha.
\end{equation*}

\begin{corollary}
  \label{cor:Besov_polydisc}
  Let $a \ge 0$ and $f: \mathbb{D}^d \to \mathbb{C}$ be holomorphic.
  Then
  \begin{equation*}
    |f(0)| + \int_0^1 \|(R f)_r\|_\infty \Big(\log \Big( \frac{1}{1-r} \Big)\Big)^a \, d r
    \simeq_a \sum_{n=0}^\infty (n+1)^a \|f * W_n\|_\infty.
  \end{equation*}
\end{corollary}

\begin{proof}
  Define $g: \mathbb{D} \to C(\mathbb{T}^d)$ by
  \begin{equation*}
    g(w)(z) = f(w z) \quad (z \in \mathbb{T}^d, w \in \mathbb{D}).
  \end{equation*}
  It is easy to check that $g$ is weakly holomorphic, hence holomorphic.
  Thus, it follows from Proposition \ref{prop:Besov_equiv} that
  \begin{equation*}
    \|g(0)\|
    + \int_0^1 \|(R g)_r\|_\infty \Big(\log \Big( \frac{1}{1-r} \Big)\Big)^a \, d r
    \simeq_a \sum_{n=0}^\infty (n+1)^a \|g * W_n\|_\infty.
  \end{equation*}
  It remains to compute  both sides. Clearly, $\|g(0)\| = |f(0)|$.
  Moreover,
  \begin{equation*}
    (R g)(w)(z) = w \frac{d}{dw} f(w z) = w \sum_{j=1}^d z_j \frac{\partial f}{\partial z_j}(w z)
    = (R f)(w z),
  \end{equation*}
  so
  \begin{equation*}
    \|(R g)_r\|_\infty = \sup_{w \in \mathbb{D}} \sup_{z \in \mathbb{T}^d} |(R g)(r w)(z)|
    = \sup_{z \in \mathbb{D}^d} | (R f)(r z)| = \|(R f)_r\|_\infty.
  \end{equation*}
  On the other hand,
  \begin{equation*}
    g(w)(z) = \sum_{\alpha} \widehat{f}(\alpha) (w z)^\alpha
    = \sum_{k=0}^\infty \sum_{|\alpha|=k} \widehat{f}(\alpha) z^\alpha w^k
  \end{equation*}
  so
  \begin{equation*}
    (g*W_n)(w)(z) = \sum_{k=0}^\infty \sum_{|\alpha|=k} \widehat{f}(\alpha) \widehat{W}_n(k)(z w)^\alpha 
    = (f*W_n)(z w)
  \end{equation*}
  and thus
  \begin{equation*}
    \| g * W_n\|_\infty = \|f* W_n\|_\infty.
  \end{equation*}
  This completes the proof.
\end{proof}

\bibliographystyle{amsplain}
\bibliography{literature}

\end{document}